\documentclass{birkmult}
\usepackage{amsfonts}
\usepackage{amsmath}
\usepackage{amssymb}
\usepackage{amsthm}
\usepackage{graphicx}%
\newtheorem{theorem}{Theorem}

\newtheorem{definition}[theorem]{Definition}

\newtheorem{lemma}[theorem]{Lemma}

\newtheorem{proposition}[theorem]{Proposition}
\newtheorem{remark}[theorem]{Remark}

\numberwithin{equation}{section}

\begin{document}
%
%
%
%
%
%
%
%
%
\title[Toeplitz operators on generalized Bergman spaces]{Toeplitz operators on generalized\\ 
Bergman spaces}
\author{Kamthorn Chailuek}

\address{%
Department of Mathematics\\
Prince of Songkla University\\
Hatyai, Songkhla, Thailand 90112}

\email{kamthorn.c@psu.ac.th}

\thanks{Supported in part by a grant from Prince of Songkla University}
\author{Brian C. Hall}
\address{%
Department of Mathematics\\
University of Notre Dame\\
255 Hurley Building\\
Notre Dame IN 46556-4618 USA}
\email{bhall@nd.edu}

\thanks{Supported in part by NSF Grant DMS-0555862}



\date{July 10, 2009}

\begin{abstract}
We consider the weighted Bergman spaces $\mathcal{H}L^{2}(\mathbb{B}^{d}%
,\mu_{\lambda}),$ where we set $d\mu_{\lambda}(z)=c_{\lambda}(1-\left\vert
z\right\vert ^{2})^{\lambda}~d\tau(z),$ with $\tau$ being the hyperbolic volume
measure. These spaces are nonzero if and only if $\lambda>d.$ For
$0<\lambda\leq d,$ spaces with the same formula for the reproducing kernel can
be defined using a Sobolev-type norm. We define Toeplitz operators on these
generalized Bergman spaces and investigate their properties. Specifically, we
describe classes of symbols for which the corresponding Toeplitz operators can
be defined as bounded operators or as a Hilbert--Schmidt operators on the
generalized Bergman spaces.

\end{abstract}

\maketitle

\section{Introduction}

\subsection{Generalized Bergman spaces}

Let $\mathbb{B}^{d}$ denote the (open) unit ball in $\mathbb{C}^{d}$ and let
$\tau$ denote the hyperbolic volume measure on $\mathbb{B}^{d},$ given by%
\begin{equation}
d\tau(z)=(1-|z|^{2})^{-(d+1)}~dz, \label{tau.form}%
\end{equation}
where $dz$ denotes the $2d$-dimensional Lebesgue measure. The measure $\tau$
is natural because it is invariant under all of the automorphisms
(biholomorphic mappings) of $\mathbb{B}^{d}.$ For $\lambda>0,$ let
$\mu_{\lambda}$ denote the measure%
\[
d\mu_{\lambda}(z)=c_{\lambda}(1-|z|^{2})^{\lambda}~d\tau(z),
\]
where $c_{\lambda}$ is a positive constant whose value will be specified
shortly. Finally, let $\mathcal{H}L^{2}(\mathbb{B}^{d},\mu_{\lambda})$ denote
the (weighted) \textbf{Bergman space}, consisting of those holomorphic
functions on $\mathbb{B}^{d}$ that are square-integrable with respect to
$\mu_{\lambda}.$ (Often these are defined using the Lebesgue measure as the
reference measure, but all the formulas look nicer if we use the hyperbolic
volume measure instead.) These spaces carry a projective unitary
representation of the group $SU(d,1).$

If $\lambda>d,$ then the measure $\mu_{\lambda}$ is finite, so that all
bounded holomorphic functions are square-integrable. For $\lambda>d,$ we
choose $c_{\lambda}$ so that $\mu_{\lambda}$ is a probability measure.
Calculation shows that%
\begin{equation}
c_{\lambda}=\frac{\Gamma(\lambda)}{\pi^{d}\Gamma(\lambda-d)},\quad\lambda>d.
\label{clambda}%
\end{equation}
(This differs from the value in Zhu's book \cite{Z} by a factor of $\pi
^{d}/d!,$ because Zhu uses normalized Lebesgue whereas we use un-normalized
Lebesgue measure in (\ref{tau.form}).) On the other hand, if $\lambda\leq d,$
then $\mu_{\lambda}$ is an infinite measure. In this case, it is not hard to
show that there are no nonzero holomorphic functions that are
square-integrable with respect to $\mu_{\lambda}$ (no matter which nonzero
value for $c_{\lambda}$ we choose).

Although the holomorphic $L^{2}$ space with respect to $\mu_{\lambda}$ is
trivial (zero dimensional) when $\lambda\leq d,$ there are indications that
life does not end at $\lambda=d.$ First, the reproducing kernel for
$\mathcal{H}L^{2}(\mathbb{B}^{d},\mu_{\lambda})$ is given by%
\[
K_{\lambda}(z,w)=\frac{1}{(1-z\cdot\bar{w})^{\lambda}}%
\]
for $\lambda>d$. The reproducing kernel is defined by the property that it is
anti-holomorphic in $w$ and satisfies
\[
\int_{\mathbb{B}^{d}}K_{\lambda}(z,w)f(w)~d\mu_{\lambda}(w)=f(z)
\]
for all $f\in\mathcal{H}L^{2}(\mathbb{B}^{d},\mu_{\lambda}).$ Nothing unusual
happens to $K_{\lambda}$ as $\lambda$ approaches $d.$ In fact, $K_{\lambda
}(z,w):=(1-z\cdot\bar{w})^{-\lambda}$ is a \textquotedblleft positive definite
reproducing kernel\textquotedblright\ for all $\lambda>0.$ Thus, it is
possible to define a reproducing kernel Hilbert space for all $\lambda>0$ that
agrees with $\mathcal{H}L^{2}(\mathbb{B}^{d},\mu_{\lambda})$ for $\lambda>d.$

Second, in representation theory, one is sometimes led to consider spaces like
$\mathcal{H}L^{2}(\mathbb{B}^{d},\mu_{\lambda})$ but with $\lambda<d.$
Consider, for example, the much-studied metaplectic representation of the
connected double cover of $SU(1,1)\cong Sp(1,\mathbb{R}).$ This representation
is a direct sum of two irreducible representations, one of which can be
realized in the Bergman space $\mathcal{H}L^{2}(\mathbb{B}^{1},\mu_{3/2})$ and
the other of which can be realized in (a suitably defined version of) the
Bergman space $\mathcal{H}L^{2}(\mathbb{B}^{1},\mu_{1/2}).$ To be precise, we
can say that the second summand of the metaplectic representation is realized
in a Hilbert space of holomorphic functions having $K_{\lambda},$
$\lambda=1/2,$ as its reproducing kernel. See \cite[Sect. 4.6]{F}.

Last, one often wants to consider the infinite-dimensional ($d\rightarrow
\infty$) limit of the spaces $\mathcal{H}L^{2}(\mathbb{B}^{d},\mu_{\lambda}).$
(See, for example, \cite{RT} and \cite{KRT}.) To do this, one wishes to embed
each space $\mathcal{H}L^{2}(\mathbb{B}^{d},\mu_{\lambda})$ isometrically into
a space of functions on $\mathbb{B}^{d+1},$ as functions that are independent
of $z_{n+1}.$ It turns out that if one uses (as we do) hyperbolic volume
measure as the reference measure, then the desired isometric embedding is
achieved by embedding $\mathcal{H}L^{2}(\mathbb{B}^{d},\mu_{\lambda})$ into
$\mathcal{H}L^{2}(\mathbb{B}^{d+1},\mu_{\lambda}).$ That is, if we use the
\textit{same value of} $\lambda$ on $\mathbb{B}^{d+1}$ as on $\mathbb{B}^{d},$
then the norm of a function $f(z_{1},\ldots,z_{d})$ is the same whether we
view it as a function on $\mathbb{B}^{d}$ or as a function on $\mathbb{B}%
^{d+1}$ that is independent of $z_{d+1}.$ (See, for example, Theorem
\ref{bergmansobolev.thm}, where the inner product of $z^{m}$ with $z^{n}$ is
independent of $d$.) If, however, we keep $\lambda$ constant as $d$ tends to
infinity, then we will eventually violate the condition $\lambda>d.$

Although it is possible to describe the Bergman spaces for $\lambda\leq d$ as
reproducing kernel Hilbert spaces, this is not the most convenient description
for calculation. Instead, drawing on several inter-related results in the
literature, we describe these spaces as \textquotedblleft holomorphic Sobolev
spaces,\textquotedblright\ also called Besov spaces. The inner product on
these spaces, which we denote as $H(\mathbb{B}^{d},\lambda),$ is an $L^{2}$
inner product involving both the functions and \textit{derivatives} of the
functions. For $\lambda>d,$ $H(\mathbb{B}^{d},\lambda)$ is identical to
$\mathcal{H}L^{2}(\mathbb{B}^{d},\mu_{\lambda})$ (the same space of functions
with the same inner product), but $H(\mathbb{B}^{d},\lambda)$ is defined for
all $\lambda>0.$

It is worth mentioning that in the borderline case $\lambda=d,$ the space
$H(\mathbb{B}^{d},\lambda)$ can be identified with the Hardy space of
holomorphic functions that are square-integrable over the boundary. To see
this, note that the normalization constant $c_{\lambda}$ tends to zero as
$\lambda$ approaches $d$ from above. Thus, the measure of any compact subset
of $\mathbb{B}^{d}$ tends to zero as $\lambda\rightarrow d^{+},$ meaning that
most of the mass of $\mu_{\lambda}$ is concentrated near the boundary. As
$\lambda\rightarrow d^{+},$ $\mu_{\lambda}$ converges, in the weak-$\ast$
topology on $\overline{\mathbb{B}^{d}},$ to the unique rotationally invariant
probability measure on the boundary. Alternatively, we may observe that the
formula for the inner product of monomials in $H(\mathbb{B}^{d},d)$ (Theorem
\ref{bergmansobolev.thm} with $\lambda=d$) is the same as in the Hardy space.

\subsection{Toeplitz operators}

One important aspect of Bergman spaces is the theory of Toeplitz operators on
them. If $\phi$ is a bounded measurable function, the we can define the
\textbf{Toeplitz operator} $T_{\phi}$ on $\mathcal{H}L^{2}(\mathbb{B}^{d}%
,\mu_{\lambda})$ by $T_{\phi}f=P_{\lambda}(\phi f),$ where $P_{\lambda}$ is
the orthogonal projection from $L^{2}(\mathbb{B}^{d},\mu_{\lambda})$ onto the
holomorphic subspace. That is, $T_{\phi}$ consists of multiplying a
holomorphic function by $\phi,$ followed by projection back into the
holomorphic subspace. Of course, $T_{\phi}$ depends on $\lambda,$ but we
suppress this dependence in the notation. The function $\phi$ is called the
(Toeplitz) \textbf{symbol} of the operator $T_{\phi}.$ The map sending $\phi$
to $T_{\phi}$ is known as the Berezin--Toeplitz quantization map and it (and
various generalizations) have been much studied. See, for example, the early
work of Berezin \cite{B1,B2}, which was put into a general framework in
\cite{Ra,RCG}, along with \cite{KL,BLU,BMS,Co}, to mention just a few works.
The Berezin--Toeplitz quantization may be thought of as a generalization of
the anti-Wick-ordered quantization on $\mathbb{C}^{d}$ (see \cite{H}).

When $\lambda<d,$ the inner product on $H(\mathbb{B}^{d},\lambda)$ is not an
$L^{2}$ inner product, and so the \textquotedblleft multiply and
project\textquotedblright\ definition of $T_{\phi}$ no longer makes sense. Our
strategy is to find alternative formulas for computing $T_{\phi}$ in the case
$\lambda>d,$ with the hope that these formulas will continue to make sense
(for certain classes of symbols $\phi$) for $\lambda\leq d.$ Specifically, we
will identify classes of symbols $\phi$ for which $T_{\phi}$ can be defined as:

\begin{itemize}
\item A bounded operator on $H(\mathbb{B}^{d},\lambda)$ (Section \ref{bddop})

\item A Hilbert--Schmidt operator on $H(\mathbb{B}^{d},\lambda)$ (Section
\ref{hsop}).
\end{itemize}

\noindent We also consider in Section \ref{poly.sec} Toeplitz operators whose
symbols are polynomials in $z$ and $\bar{z}$ and observe some unusual
properties of such operators in the case $\lambda<d.$

\subsection{Acknowledgments}

The authors thank M. Engli\v{s} for pointing out to them several useful
references and B. Driver for useful suggestions regarding the results in
Section \ref{bddop}. This article is an expansion of the Ph.D. thesis of the
first author, written under the supervision of the second author. We also
thank the referee for helpful comments and corrections.

\section{$H(\mathbb{B}^{d},\lambda)$ as a holomorphic Sobolev
space\label{bergmansobolev.sec}}

In this section, we construct a Hilbert space of holomorphic functions on
$\mathbb{B}^{d}$ with reproducing kernel $(1-z\cdot\bar{w})^{-\lambda},$ for
an arbitrary $\lambda>0.$ We denote this space as $H(\mathbb{B}^{d},\lambda).$
The inner product on this space is an $L^{2}$ inner product with respect to
the measure $\mu_{\lambda+2n},$ where $n$ is chosen so that $\lambda+2n>d.$
The inner product, however, involves not only the holomorphic functions but
also their derivatives. That is, $H(\mathbb{B}^{d},\lambda)$ is a sort of
holomorphic Sobolev space (or Besov space) with respect to the measure
$\mu_{\lambda+2n}.$ When $\lambda>d,$ our space is identical to $\mathcal{H}%
L^{2}(\mathbb{B}^{d},\mu_{\lambda})$---not just the same space of functions,
but also the same inner product. When $\lambda\leq d,$ the Hilbert space
$H(\mathbb{B}^{d},\lambda),$ with the associated projective unitary action of
$SU(d,1),$ is sometimes referred to as the analytic continuation (with respect
to $\lambda$) of the holomorphic discrete series.

Results in the same spirit as---and in some cases almost identical to---the
results of this section have appeared in several earlier works, some of which
treat arbitrary bounded symmetric domains and not just the ball in
$\mathbb{C}^{d}.$ For example, in the case of the unit ball in $\mathbb{C}%
^{d},$ Theorem 3.13 of \cite{Ya} would presumably reduce to almost the same
expression as in our Theorem \ref{bergmansobolev.thm}, except that Yan has all
the derivatives on one side, in which case the inner product has to be
interpreted as a limit of integrals over a ball of radius $1-\varepsilon.$
(Compare the formula for $\mathcal{D}_{\lambda}^{k}$ on p. 13 of \cite{Ya} to
the formula for $A$ and $B$ in Theorem \ref{bergmansobolev.thm}.) See also
\cite{Ar,BB,Ka,ZZ,Z1}. Note, however, a number of these references give a
construction that yields, for $\lambda>d,$ the same space of functions as
$\mathcal{H}L^{2}(\mathbb{B}^{d},\mu_{\lambda})$ with a different but
equivalent norm. Such an approach is not sufficient for our needs; we require
the \textit{same inner product} as well as the same space of functions.

Although our results in this section are not really new, we include proofs to
make the paper self-contained and to get the precise form of the results that
we want. The integration-by-parts argument we use also serves to prepare for
our definition of Toeplitz operators on $H(\mathbb{B}^{d},\lambda)$ in Section
\ref{bddop}. We ourselves were introduced to this sort of reasoning by the
treatment in Folland's book \cite{F} of the disk model for the metaplectic
representation. The paper \cite{HL} obtains results in the same spirit as
those of this section, but in the context of a complex semisimple Lie group.

We begin by showing that for $\lambda>d,$ the space $\mathcal{H}%
L^{2}(\mathbb{B}^{d},\mu_{\lambda})$ can be expressed as a subspace of
$\mathcal{H}L^{2}(\mathbb{B}^{d},\mu_{\lambda+2n})$, with a Sobolev-type norm,
for any positive integer $n.$ Let $N$ denote the \textquotedblleft number
operator,\textquotedblright\ defined by%
\[
N=\sum_{j=1}^{d}z_{j}\frac{\partial}{\partial z_{j}}.
\]
This operator satisfies $Nz^{m}=|m|z^{m}$ for all multi-indices $m.$ If $f$ is
holomorphic, then $Nf$ coincides with the \textquotedblleft radial
derivative\textquotedblright\ $\left.  df(rz)/dr\right\vert _{r=1}.$ We use
also the operator $\bar{N}=\sum_{j=1}^{d}\bar{z}_{j}\partial/\partial\bar
{z}_{j}.$

A simple computation shows that
\begin{equation}
(1-|z|^{2})^{\alpha}=\left(  I-\frac{N}{\alpha+1}\right)  (1-|z|^{2}%
)^{\alpha+1}=\left(  I-\frac{\bar{N}}{\alpha+1}\right)  (1-|z|^{2})^{\alpha
+1}. \label{nIdent}%
\end{equation}

We will use (\ref{nIdent}) and the following integration by parts result,
which will also be used in Section \ref{bddop}.

\begin{lemma}
\label{parts.lem}If $\lambda>d$ and $\psi$ is a continuously differentiable
function for which $\psi$ and $N\psi$ are bounded, then%
\begin{align*}
c_{\lambda}\int_{\mathbb{B}^{d}}\psi(z)(1-|z|^{2})^{\lambda-d-1}dz  &
=c_{\lambda+1}\int_{\mathbb{B}^{d}}\left[  \left(  I+\frac{N}{\lambda}\right)
\psi\right]  (z)(1-|z|^{2})^{\lambda-d}~dz\\
&  =c_{\lambda+1}\int_{\mathbb{B}^{d}}\left[  \left(  I+\frac{\bar{N}}%
{\lambda}\right)  \psi\right]  (z)(1-|z|^{2})^{\lambda-d}~dz.
\end{align*}
Here $dz$ is the $2d$-dimensional Lebesgue measure on $\mathbb{B}^{d}.$
\end{lemma}

\begin{proof}
We start by applying (\ref{nIdent}) and then think of the integral over
$\mathbb{B}^{d}$ as the limit as $r$ approaches 1 of the integral over a ball
of radius $r<1.$ On the ball of radius $r,$ we write out $\partial/\partial
z_{j}$ in terms of $\partial/\partial x_{j}$ and $\partial/\partial y_{j}.$
For, say, the $\partial/\partial x_{j}$ term we express the integral as a
one-dimensional integral with respect to $x_{j}$ (with limits of integration
depending on the other variables) followed by an integral with respect to the
other variables. We then use ordinary integration by parts in the $x_{j}$
integral, and similarly for the $\partial/\partial y_{j}$ term.

The integration by parts will yield a boundary term involving $z_{j}%
\psi(z)(1-|z|^{2})^{\lambda-d}$; this boundary term will vanish as $r$ tends
to 1, because we assume $\lambda>d.$ In the nonboundary term, the operator $N$
applied to $(1-|z|^{2})^{\lambda-d}$ will turn into the operator $-\sum
_{j=1}^{d}\partial/\partial z_{j}\circ z_{j}=-(dI+N)$ applied to $\psi.$
Computing from (\ref{clambda}) that $c_{\lambda}/c_{\lambda+1}=(\lambda
-d)/\lambda$, we may simplify and let $r$ tend to 1 to obtain the desired
result involving $N.$ The same reasoning gives the result involving $\bar{N}$
as well.
\end{proof}

We now state the key result, obtained from (\ref{nIdent}) and Lemma
\ref{parts.lem}, relating the inner product in $\mathcal{H}L^{2}%
(\mathbb{B}^{d},\mu_{\lambda})$ to the inner product in $\mathcal{H}%
L^{2}(\mathbb{B}^{d},\mu_{\lambda+1})$ (compare \cite[p. 215]{F} in the case
$d=1$).

\begin{proposition}
\label{shift1.prop}Suppose that $\lambda>d$ and $f$ and $g$ are holomorphic
functions on $\mathbb{B}^{d}$ for which $f,$ $g,$ $Nf,$ and $Ng$ are all
bounded. Then%
\begin{equation}
\left\langle f,g\right\rangle _{L^{2}(\mathbb{B}^{d},\mu_{\lambda}%
)}=\left\langle f,\left(  I+\frac{N}{\lambda}\right)  g\right\rangle
_{L^{2}(\mathbb{B}^{d},\mu_{\lambda+1})}=\left\langle \left(  I+\frac
{N}{\lambda}\right)  f,g\right\rangle _{L^{2}(\mathbb{B}^{d},\mu_{\lambda+1}%
)}. \label{shift1}%
\end{equation}

\end{proposition}

\begin{proof}
Recalling the formula (\ref{tau.form}) for the measure $\tau,$ we apply Lemma
\ref{parts.lem} with $\psi(z)=\overline{f(z)}g(z)$ with $f$ and $g$
holomorphic. Observing that $N(\bar{f}g)=\bar{f}Ng$ gives the first equality
and observing that $\bar{N}(\bar{f}g)=\overline{(Nf)}g$ gives the second equality.
\end{proof}

Now, a general function in $\mathcal{H}L^{2}(\mathbb{B}^{d},\mu_{\lambda})$ is
not bounded. Indeed, the pointwise bounds on elements of $\mathcal{H}%
L^{2}(\mathbb{B}^{d},\mu_{\lambda}),$ coming from the reproducing kernel, are
not sufficient to give a direct proof of the vanishing of the boundary terms
in the integration by parts in Proposition \ref{shift1.prop}. Nevertheless,
(\ref{shift1}) does hold for all $f$ and $g$ in $\mathcal{H}L^{2}%
(\mathbb{B}^{d},\mu_{\lambda}),$ \textit{provided} that one interprets the
inner product as the limit as $r$ approaches 1 of integration over a ball of
radius $r.$ (See \cite[p. 215]{F} or \cite[Thm. 3.13]{Ya}.) We are going to
iterate (\ref{shift1}) to obtain an expression for the inner product on
$\mathcal{H}L^{2}(\mathbb{B}^{d},\mu_{\lambda})$ involving equal numbers of
derivatives on $f$ and $g.$ This leads to the following result.

\begin{theorem}
\label{shift2n.thm}Fix $\lambda>d$ and a non-negative integer $n.$ Then a
holomorphic function $f$ on $\mathbb{B}^{d}$ belongs to $\mathcal{H}%
L^{2}(\mathbb{B}^{d},\mu_{\lambda})$ if and only if $N^{l}f$ belongs to
$\mathcal{H}L^{2}(\mathbb{B}^{d},\mu_{\lambda+2n})$ for $0\leq l\leq n.$
Furthermore,%
\begin{equation}
\left\langle f,g\right\rangle _{\mathcal{H}L^{2}(\mathbb{B}^{d},\mu_{\lambda
})}=\left\langle Af,Bg\right\rangle _{\mathcal{H}L^{2}(\mathbb{B}^{d}%
,\mu_{\lambda+2n})} \label{shift2n}%
\end{equation}
for all $f,g\in\mathcal{H}L^{2}(\mathbb{B}^{d},\mu_{\lambda}),$ where%
\begin{align*}
A  &  =\left(  I+\frac{N}{\lambda+n}\right)  \left(  I+\frac{N}{\lambda
+n+1}\right)  \cdots\left(  I+\frac{N}{\lambda+2n-1}\right) \\
B  &  =\left(  I+\frac{N}{\lambda}\right)  \left(  I+\frac{N}{\lambda
+1}\right)  \cdots\left(  I+\frac{N}{\lambda+n-1}\right)  .
\end{align*}

\end{theorem}

Let us make a few remarks about this result before turning to the proof. Let
$\sigma=\lambda+2n.$ It is not hard to see that $N^{k}f$ belongs to
$\mathcal{H}L^{2}(\mathbb{B}^{d},\mu_{\sigma})$ for $0\leq k\leq n$ if and
only if all the partial derivatives of $f$ up to order $n$ belong to
$\mathcal{H}L^{2}(\mathbb{B}^{d},\mu_{\mu}),$ so we may describe this
condition as \textquotedblleft$f$ has $n$ derivatives in $\mathcal{H}%
L^{2}(\mathbb{B}^{d},\mu_{\sigma}).$\textquotedblright\ This condition then
implies that $f$ belongs to $\mathcal{H}L^{2}(\mathbb{B}^{d},\mu_{\sigma
-2n}),$ which in turn means that $f(z)/(1-\left\vert z\right\vert ^{2})^{n}$
belongs to $L^{2}(\mathbb{B}^{d},\mu_{\sigma}).$ Since $1/(1-\left\vert
z\right\vert ^{2})^{n}$ blows up at the boundary of $\mathbb{B}^{d},$ saying
that $f(z)/(1-\left\vert z\right\vert ^{2})^{n}$ belongs to $L^{2}%
(\mathbb{B}^{d},\mu_{\sigma})$ says that $f(z)$ has better behavior at the
boundary than a typical element of $\mathcal{H}L^{2}(\mathbb{B}^{d}%
,\mu_{\sigma}).$ We may summarize this discussion by saying that each
derivative that $f\in\mathcal{H}L^{2}(\mathbb{B}^{d},\mu_{\sigma})$ has in
$\mathcal{H}L^{2}(\mathbb{B}^{d},\mu_{\sigma})$ results, roughly speaking, in
an improvement by a factor of $(1-\left\vert z\right\vert ^{2})$ in the
behavior of $f$ near the boundary.

This improvement is also reflected in the pointwise bounds on $f$ coming from
the reproducing kernel. If $f$ has $n$ derivatives in $\mathcal{H}%
L^{2}(\mathbb{B}^{d},\mu_{\sigma}),$ then $f$ belongs to $\mathcal{H}%
L^{2}(\mathbb{B}^{d},\mu_{\sigma-2n}),$ which means that $f$ satisfies the
pointwise bounds
\begin{align}
\left\vert f(z)\right\vert  &  \leq\left\Vert f\right\Vert _{L^{2}%
(\mathbb{B}^{d},\mu_{\sigma-2n})}\left(  K_{\sigma-2n}(z,z)\right)
^{1/2}\nonumber\\
&  =\left\Vert f\right\Vert _{L^{2}(\mathbb{B}^{d},\mu_{\sigma-2n})}\left(
\frac{1}{1-\left\vert z\right\vert ^{2}}\right)  ^{\frac{\sigma}{2}-n}.
\label{improve.bounds}%
\end{align}
These bounds are better by a factor of $(1-\left\vert z\right\vert ^{2})^{n}$
than the bounds on a typical element of $\mathcal{H}L^{2}(\mathbb{B}^{d}%
,\mu_{\sigma}).$ See also \cite{HL} for another setting in which the existence
of derivatives in a holomorphic $L^{2}$ space can be related in a precise way
to improved pointwise behavior of the functions.

The results of the two previous paragraphs were derived under the assumption
that $\lambda=\sigma-2n>d.$ However, Theorem \ref{bergmansobolev.thm} will
show that (\ref{improve.bounds}) still holds under the assumption
$\lambda=\sigma-2n>0.$

\begin{proof}
If $f$ and $g$ are polynomials, then (\ref{shift2n}) follows from iteration of
Proposition \ref{shift1.prop}. Note that $N$ is a non-negative operator on
polynomials, because the monomials form an orthogonal basis of eigenvectors
with non-negative eigenvalues. It is well known and easily verified that for
any $f$ in $\mathcal{H}L^{2}(\mathbb{B}^{d},\mu_{\lambda}),$ the partial sums
of the Taylor series of $f$ converge to $f$ in norm. We can therefore choose
polynomials $f_{j}$ converging in norm to $f$. If we apply (\ref{shift2n})
with $f=g=(f_{j}-f_{k})$ and expand out the expressions for $A$ and $B,$ then
the positivity of $N$ will force each of the terms on the right-hand side to
tend to zero. In particular, $N^{l}f_{j}$ is a Cauchy sequences in
$\mathcal{H}L^{2}(\mathbb{B}^{d},\mu_{\lambda+2n}),$ for all $0\leq l\leq n.$
It is easily seen that the limit of this sequence is $N^{l}f$; for holomorphic
functions, $L^{2}$ convergence implies locally uniform convergence of the
derivatives to the corresponding derivatives of the limit function. This shows
that $N^{l}f$ is in $\mathcal{H}L^{2}(\mathbb{B}^{d},\mu_{\lambda+2n}).$ For
any $f,g\in\mathcal{H}L^{2}(\mathbb{B}^{d},\mu_{\lambda}),$ choose sequences
$f_{j}$ and $g_{k}$ of polynomials converging to $f,g.$ Since $N^{l}f_{j}$ and
$N^{l}g_{j}$ converge to $N^{l}f$ and $N^{l}g,$ respectively, plugging $f_{j}$
and $g_{j}$ into (\ref{shift2n}) and taking a limit gives (\ref{shift2n}) in general.

In the other direction, suppose that $N^{l}f$ belongs to $\mathcal{H}%
L^{2}(\mathbb{B}^{d},\mu_{\lambda+2n})$ for all $0\leq l\leq n.$ Let $f_{j}$
denote the $j$th partial sum of the Taylor series of $f$. Then since
$Nz^{m}=|m|z^{m}$ for all multi-indices $m,$ the functions $N^{l}f_{j}$ form
the partial sums of a Taylor series converging to $N^{l}f_{j},$ and so these
must be the partial sums of \textit{the} Taylor series of $N^{l}f.$ Thus, for
each $l,$ we have that $N^{l}f_{j}$ converges to $N^{l}f$ in $\mathcal{H}%
L^{2}(\mathbb{B}^{d},\mu_{\lambda+2n}).$ If we then apply (\ref{shift2n}) with
$f=g=f_{j}-f_{k},$ convergence of each $N^{l}f_{j}$ implies that all the terms
on the right-hand side tend to zero. We conclude that $f_{j}$ is a Cauchy
sequence in $\mathcal{H}L^{2}(\mathbb{B}^{d},\mu_{\lambda}),$ which converges
to some $\hat{f}.$ But $L^{2}$ convergence of holomorphic functions implies
pointwise convergence, so the limit in $\mathcal{H}L^{2}(\mathbb{B}^{d}%
,\mu_{\lambda})$ (i.e., $\hat{f}$) coincides with the limit in $\mathcal{H}%
L^{2}(\mathbb{B}^{d},\mu_{\lambda+2n})$ (i.e., $f$). This shows that $f$ is in
$\mathcal{H}L^{2}(\mathbb{B}^{d},\mu_{\lambda}).$
\end{proof}

Now, when $\lambda\leq d,$ Proposition \ref{shift1} no longer holds. This is
because the boundary terms, which involve $(1-|z|^{2})^{\lambda-d},$ no longer
vanish. This failure of equality is actually a \textit{good} thing, because if
we take $f=g$, then
\[
c_{\lambda}\int_{\mathbb{B}^{d}}\left\vert f\left(  z\right)  \right\vert
^{2}(1-|z|^{2})^{\lambda}~d\tau(z)=+\infty
\]
for all nonzero holomorphic functions, no matter what positive value we assign
to $c_{\lambda}.$ (Recall that when $\lambda>d,$ $c_{\lambda}$ is chosen to
make $\mu_{\lambda}$ a probability measure, but this prescription does not
make sense for $\lambda\leq d.$) Although the left-hand side of (\ref{shift1})
is infinite when $f=g$ and $\lambda\leq d,$ the right-hand side is finite if
$\lambda+1>d$ and, say, $f$ is a polynomial.

More generally, for any $\lambda\leq d,$ we can choose $n$ big enough that
$\lambda+2n>d.$ We then take the right-hand side of (\ref{shift2n}) as a definition.

\begin{theorem}
\label{bergmansobolev.thm}For all $\lambda>0,$ choose a non-negative integer
$n$ so that $\lambda+2n>d$ and define%
\[
H(\mathbb{B}^{d},\lambda)=\left\{  f\in\mathcal{H}(\mathbb{B}^{d})\left\vert
N^{k}f\in\mathcal{H}L^{2}(\mathbb{B}^{d},\mu_{\lambda+2n}),~0\leq k\leq
n\right.  \right\}  .
\]
Then the formula%
\[
\left\langle f,g\right\rangle _{\lambda}=\left\langle Af,Bg\right\rangle
_{\mathcal{H}L^{2}(\mathbb{B}^{d},\mu_{\lambda+2n})}%
\]
where%
\begin{align*}
A  &  =\left(  I+\frac{N}{\lambda+n}\right)  \left(  I+\frac{N}{\lambda
+n+1}\right)  \cdots\left(  I+\frac{N}{\lambda+2n-1}\right) \\
B  &  =\left(  I+\frac{N}{\lambda}\right)  \left(  I+\frac{N}{\lambda
+1}\right)  \cdots\left(  I+\frac{N}{\lambda+n-1}\right)
\end{align*}
defines an inner product on $H(\mathbb{B}^{d},\lambda)$ and $H(\mathbb{B}%
^{d},\lambda)$ is complete with respect to this inner product.

The monomials $z^{m}$ form an orthogonal basis for $H(\mathbb{B}^{d},\lambda)$
and for all multi-indices $l$ and $m$ we have%
\[
\left\langle z^{l},z^{m}\right\rangle _{\lambda}=\delta_{l,m}\frac
{m!\Gamma(\lambda)}{\Gamma(\lambda+|m|)}.
\]
Furthermore, $H(\mathbb{B}^{d},\lambda)$ has a reproducing kernel given by%
\[
K_{\lambda}(z,w)=\frac{1}{(1-z\cdot\bar{w})^{\lambda}}.
\]

\end{theorem}

Using power series, it is easily seen that for any holomorphic function $f,$
if $N^{n}f$ belongs to $\mathcal{H}L^{2}(\mathbb{B}^{d},\mu_{\lambda+2n}),$
then $N^{k}f$ belongs to $\mathcal{H}L^{2}(\mathbb{B}^{d},\mu_{\lambda+2n})$
for $0\leq k\,<n.$

Note that the reproducing kernel and the inner product of the monomials are
independent of $n.$ Thus, we obtain the same space of functions with the same
inner product, no matter which $n$ we use, so long as $\lambda+2n>d.$

From the reproducing kernel we obtain the pointwise bounds given by $\left\vert
f(z)\right\vert ^{2}\leq\left\Vert f\right\Vert _{\lambda}^{2}(1-\left\vert
z\right\vert ^{2})^{-\lambda}.$

\begin{proof}
Using a power series argument, it is easily seen that if $f$ and $N^{k}f$
belong to $\mathcal{H}L^{2}(\mathbb{B}^{d},\mu_{\lambda+2n})$, then
$\left\langle f,N^{k}f\right\rangle _{L^{2}(\mathbb{B}^{d},\mu_{\lambda+2n}%
)}\geq0.$ From this, we obtain positivity of the inner product $\left\langle
\cdot,\cdot\right\rangle _{\lambda}.$ If $f_{j}$ is a Cauchy sequence in
$H(\mathbb{B}^{d},\lambda),$ then positivity of the coefficients in the
expressions for $A$ and $B$ imply that for $0\leq k\leq n,$ $N^{k}f_{j}$ is a
Cauchy sequence in $\mathcal{H}L^{2}(\mathbb{B}^{d},\mu_{\lambda+2n}),$ which
converges (as in the proof of Theorem \ref{shift2n.thm}) to $N^{k}f.$ This
shows that $N^{k}f$ is in $\mathcal{H}L^{2}(\mathbb{B}^{d},\mu_{\lambda+2n})$
for each $0\leq k\leq n,$ and so $f\in H(\mathbb{B}^{d},\lambda).$ Further,
convergence of each $N^{k}f_{j}$ to $N^{k}f$ implies that $f_{j}$ converges to
$f$ in $H(\mathbb{B}^{d},\lambda).$

To compute the inner product of two monomials in $H(\mathbb{B}^{d},\lambda),$
we apply the definition. Since $Nz^{m}=|m|z^{m},$ we obtain%
\begin{align*}
&  \left\langle z^{l},z^{m}\right\rangle _{\lambda}\\
&  =\delta_{l,m}\left(  \frac{\lambda+|m|}{\lambda}\right)  \left(
\frac{\lambda+1+|m|}{\lambda+1}\right)  \cdots\left(  \frac{\lambda
+2n-1+|m|}{\lambda+2n-1}\right)  \frac{m!\Gamma(\lambda+2n)}{\Gamma
(\lambda+2n+|m|)}\\
&  =\delta_{l,m}\frac{m!\Gamma(\lambda)}{\Gamma(\lambda+|m|)},
\end{align*}
where we have used the known formula for the inner product of monomials in
$\mathcal{H}L^{2}(\mathbb{B}^{d},\mu_{\lambda+2n})$ (e.g., \cite{Z}).

Completeness of the monomials holds in $H(\mathbb{B}^{d},\lambda)$ for
essentially the same reason it holds in the ordinary Bergman spaces. For
$f\in\mathcal{H}L^{2}(\mathbb{B}^{d},\mu_{\lambda})$, expand $f$ in a Taylor
series and then consider $\left\langle z^{m},f\right\rangle _{\lambda}$. Each
term in the inner product is an integral over $\mathbb{B}^{d}$ with respect to
$\mu_{\lambda+2n}$, and each of these integrals can be computed as the limit
as $r$ tends to 1 of integrals over a ball of radius $r<1.$ On the ball of
radius $r,$ we may interchange the integral with the sum in the Taylor series.
But distinct monomials are orthogonal not just over $\mathbb{B}^{d}$ but also
over the ball of radius $r,$ as is easily verified. The upshot of all of this
is that $\left\langle z^{m},f\right\rangle _{\lambda}$ is a nonzero multiple
of the $m$th Taylor coefficient of $f.$ Thus if $\left\langle z^{m}%
,f\right\rangle _{\lambda}=0$ for all $m,$ $f$ is identically zero.

Finally, we address the reproducing kernel. Although one can use essentially
the same argument as in the case $\lambda>d,$ using the orthogonal basis of
monomials and a binomial expansion (see the proof of Theorem \ref{hsl2.thm}),
it is more enlightening to relate the reproducing kernel in $H(\mathbb{B}%
^{d},\lambda)$ to that in $\mathcal{H}L^{2}(\mathbb{B}^{d},\mu_{\lambda+2n}).$
We require some elementary properties of the operators $A$ and $B$; since the
monomials form an orthogonal basis of eigenvectors for these operators, these
properties are easily obtained. We need that $A$ is self-adjoint on its
natural domain and that $A$ and $B$ have bounded inverses.

Let $\chi_{z}^{\lambda+2n}$ be the unique element of $\mathcal{H}%
L^{2}(\mathbb{B}^{d},\mu_{\lambda+2n})$ for which
\[
\left\langle \chi_{z}^{\lambda+2n},f\right\rangle _{L^{2}(\mathbb{B}^{d}%
,\mu_{\lambda+2n})}=f(z)
\]
for all $f$ in $\mathcal{H}L^{2}(\mathbb{B}^{d},\mu_{\lambda+2n}).$
Explicitly, $\chi_{z}^{\lambda+2n}(w)=(1-\bar{z}\cdot w)^{-(\lambda+2n)}.$
(This is Theorem 2.2 of \cite{Z} with our $\lambda$ corresponding to
$n+\alpha+1$ in \cite{Z}.) Now, a simple calculation shows that%
\begin{equation}
(I+N/a)(1-\bar{z}\cdot w)^{-a}=(1-\bar{z}\cdot w)^{-(a+1)}, \label{chia}%
\end{equation}
where $N$ acts on the $w$ variable with $z$ fixed. From this, we see that
$N^{k}\chi_{z}^{\lambda+2n}$ is a bounded function for each fixed
$z\in\mathbb{B}^{d}$ and $k\in\mathbb{N},$ so that $\chi_{z}^{\lambda+2n}$ is
in $H(\mathbb{B}^{d},\lambda).$

For any $f\in H(\mathbb{B}^{d},\lambda)$ we compute that%
\begin{align*}
\left\langle f,(AB)^{-1}\chi_{z}^{\lambda+2n}\right\rangle _{\lambda}  &
=\left\langle Af,B(AB)^{-1}\chi_{z}^{\lambda+2n}\right\rangle _{L^{2}%
(\mathbb{B}^{d},\mu_{\lambda+2n)}}\\
&  =\left\langle f,\chi_{z}^{\lambda+2n}\right\rangle _{L^{2}(\mathbb{B}%
^{d},\mu_{\lambda+2n)}}=f(z).
\end{align*}
This shows that the reproducing kernel for $H(\mathbb{B}^{d},\lambda)$ is
given by $K_{\lambda}(z,w)=\overline{[(AB)^{-1}\chi_{z}^{\lambda+2n}](w)}.$
Using (\ref{chia}) repeatedly gives the desired result.
\end{proof}

We conclude this section with a simple lemma that will be useful in Section
\ref{hsop}.

\begin{lemma}
\label{2lambda.lem}For all $\lambda_{1},\lambda_{2}>0,$ if $f$ is in
$H(\mathbb{B}^{d},\lambda_{1})$ and $g$ is in $H(\mathbb{B}^{d},\lambda_{2})$
then $fg$ is in $H(\mathbb{B}^{d},\lambda_{1}+\lambda_{2}).$
\end{lemma}

\begin{proof}
If, say, $\lambda_{1}>d,$ then we have the following simple argument:%
\begin{align*}
\left\Vert fg\right\Vert _{\lambda_{1}+\lambda_{2}}^{2}  &  =c_{\lambda
_{1}+\lambda_{2}}\int_{\mathbb{B}^{d}}\left\vert f(z)\right\vert
^{2}\left\vert g(z)\right\vert ^{2}(1-\left\vert z\right\vert ^{2}%
)^{\lambda_{1}+\lambda_{2}}~d\tau(z)\\
&  \leq c_{\lambda_{1}+\lambda_{2}}\left\Vert g\right\Vert _{\lambda_{2}}%
^{2}\int_{\mathbb{B}^{d}}\left\vert f(z)\right\vert ^{2}(1-\left\vert
z\right\vert ^{2})^{-\lambda_{2}}(1-\left\vert z\right\vert ^{2})^{\lambda
_{1}+\lambda_{2}}~d\tau(z)\\
&  =\frac{c_{\lambda_{1}+\lambda_{2}}}{c_{\lambda_{1}}}\left\Vert f\right\Vert
_{\lambda_{1}}^{2}\left\Vert g\right\Vert _{\lambda_{2}}^{2}.
\end{align*}
Unfortunately, $c_{\lambda_{1}+\lambda_{2}}/c_{\lambda_{1}}$ tends to infinity
as $\lambda_{1}$ approaches $d$ from above, so we cannot expect this simple
inequality to hold for $\lambda_{1}<d.$

For any $\lambda_{1},\lambda_{2}>0,$ choose $n$ so that $\lambda_{1}+n>d$ and
$\lambda_{2}+n>d.$ Then $fg$ belongs to $H(\mathbb{B}^{d},\lambda_{1}%
+\lambda_{2})$ provided that $N^{n}(fg)$ belongs to $\mathcal{H}%
L^{2}(\mathbb{B}^{d},\lambda_{1}+\lambda_{2}+2n).$ But%
\begin{equation}
N^{n}(fg)=\sum_{k=0}^{n}\binom{n}{k}N^{k}f~N^{n-k}g. \label{kterm}%
\end{equation}
Using Theorem \ref{bergmansobolev.thm}, it is easy to see that if $f$ belongs
to $H(\mathbb{B}^{d},\lambda_{1})$ then $N^{k}f$ belongs to $H(\mathbb{B}%
^{d},\lambda_{1}+2k)$. Thus,%
\[
\left\vert N^{k}f(z)\right\vert ^{2}\leq a_{k}(1-\left\vert z\right\vert
^{2})^{-(\lambda_{1}+2k)}.
\]

Now, for each term in (\ref{kterm}) with $k\leq n/2$, we then obtain the
following norm estimate:%
\begin{align}
&  c_{\lambda_{1}+\lambda_{2}+2n}\int_{\mathbb{B}^{d}}\left\vert
N^{k}f(z)N^{n-k}g(z)\right\vert ^{2}~(1-\left\vert z\right\vert ^{2}%
)^{\lambda_{1}+\lambda_{2}+2n}~d\tau(z)\nonumber\\
&  \leq c_{\lambda_{1}+\lambda_{2}+2n}a_{k}\int_{\mathbb{B}^{d}}\left\vert
N^{n-k}g(z)\right\vert ^{2}(1-\left\vert z\right\vert ^{2})^{\lambda
_{2}+2n-2k}~d\tau(z). \label{knorm}%
\end{align}
Since $k\leq n/2,$ we have $\lambda_{2}+2n-2k\geq\lambda_{2}+n>d.$ We are
assuming that $g$ is in $H(\mathbb{B}^{d},\lambda_{2}),$ so that $N^{n-k}g$ is
in $H(\mathbb{B}^{d},\lambda_{2}+2n-2k),$ which coincides with $\mathcal{H}%
L^{2}(\mathbb{B}^{d},\mu_{\lambda_{2}+2n-2k}).$ Thus, under our assumptions on
$f$ and $g,$ each term in (\ref{kterm}) with $k\leq n/2$ belongs to
$\mathcal{H}L^{2}(\mathbb{B}^{d},\lambda_{1}+\lambda_{2}+2n).$ A similar
argument with the roles of $f$ and $g$ reversed takes care of the terms with
$k\geq n/2.$
\end{proof}

\section{Toeplitz operators with polynomial symbols\label{poly.sec}}

In this section, we will consider our first examples of Toeplitz operators on
generalized Bergman spaces, those whose symbols are (not necessarily
holomorphic) polynomials. Such examples are sufficient to see some interesting
new phenomena, that is, properties of ordinary Toeplitz operator that fail
when extended to these generalized Bergman spaces. The definition of Toeplitz
operators for the case of polynomial symbols is consistent with the definition
we use in Section \ref{bddop} for a larger class of symbols.

For $\lambda>d,$ we define the Toeplitz operator $T_{\phi}$ by%
\[
T_{\phi}f=P_{\lambda}(\phi f)
\]
for all $f$ in $\mathcal{H}L^{2}(\mathbb{B}^{d},\mu_{\lambda})$ and all
bounded measurable functions $\phi.$ Recall that $P_{\lambda}$ is the
orthogonal projection from $L^{2}(\mathbb{B}^{d},\tau)$ onto the holomorphic
subspace. Because $P_{\lambda}$ is a self-adjoint operator on $L^{2}%
(\mathbb{B}^{d},\mu_{\lambda}),$ the matrix entries of $T_{\phi}$ may be
calculated as%
\begin{equation}
\left\langle f_{1},T_{\phi}f_{2}\right\rangle _{\mathcal{H}L^{2}%
(\mathbb{B}^{d},\mu_{\lambda})}=\left\langle f_{1},\phi f_{2}\right\rangle
_{L^{2}(\mathbb{B}^{d},\mu_{\lambda})},\quad\lambda>d, \label{matrix.entry}%
\end{equation}
for all $f_{1},f_{2}\in\mathcal{H}L^{2}(\mathbb{B}^{d},\mu_{\lambda}).$ From
this formula, it is easy to see that $T_{\bar{\phi}}=(T_{\phi})^{\ast}.$

If $\psi$ is a bounded holomorphic function and $\phi$ is any bounded
measurable function, then it is easy to see that $T_{\phi\psi}=T_{\phi}%
M_{\psi}.$ Thus, for any two multi-indices $m$ and $n,$ we have%
\begin{equation}
T_{\bar{z}^{m}z^{n}}=(M_{z^{m}})^{\ast}(M_{z^{n}}). \label{poly.def}%
\end{equation}
We will take (\ref{poly.def}) as a definition for $0<\lambda\leq d.$ Our first
task, then, is to show that $M_{z^{n}}$ is a bounded operator on
$\mathcal{H}L^{2}(\mathbb{B}^{d},\mu_{\lambda})$ for all $\lambda>0.$

\begin{proposition}
For all $\lambda>0$ and all multi-indices $n,$ the multiplication operator
$M_{z^{n}}$ is a bounded operator on $H(\mathbb{B}^{d},\lambda).$ Thus, for
any polynomial $\phi,$ the Toeplitz operator $T_{\phi}$ defined in
(\ref{poly.def}) is a bounded operator on $H(\mathbb{B}^{d},\lambda).$
\end{proposition}

\begin{proof}
The result is a is a special case of a result of Arazy and Zhang \cite{AZ} and
also of the results of Section \ref{bddop}, but it is easy to give a direct
proof. It suffices to show that $M_{z_{j}}$ is bounded for each $j.$ Since
$M_{z_{j}}$ preserves the orthogonality of the monomials, we obtain%
\[
\left\Vert M_{z_{j}}\right\Vert =\sup_{m}\frac{\left\Vert z_{j}z^{m}%
\right\Vert _{\lambda}}{\left\Vert z^{m}\right\Vert _{\lambda}}=\sup_{m}%
\frac{m_{j}+1}{|m|+\lambda}.
\]
Note that $m_{j}\leq\left\vert m\right\vert $ with equality when $m_{k}=0$ for
$k\neq j.$ Thus the supremum is finite and is easily seen to have the value of
1 if $\lambda\geq1$ and $1/\lambda$ if $\lambda<1.$
\end{proof}

We now record some standard properties of Toeplitz operators on (ordinary)
Bergman spaces. These properties hold for Toeplitz operators (defined by the
\textquotedblleft multiply and project\textquotedblright\ recipe) on any
holomorphic $L^{2}$ space. We will show that these properties \textit{do not}
hold for Toeplitz operators with polynomial symbols on the generalized Bergman
spaces $H(\mathbb{B}^{d},\lambda),$ $\lambda<d.$

\begin{proposition}
\label{basic.prop}For $\lambda>d$ and $\phi(z)$ bounded, the Toeplitz operator
$T_{\phi}$ on the space $\mathcal{H}L^{2}(\mathbb{B}^{d},d\mu_{\lambda}),$ which is
defined by $T_{\phi}f=P_{\lambda}(\phi f),$ has the following properties.

\begin{enumerate}
\item $\|T_{\phi}\|\leq\sup_{z}|\phi(z)|$

\item If $\phi(z)\geq0$ for all $z,$ then $T_{\phi}$ is a positive operator.
\end{enumerate}

Both of these properties fail when $\lambda<d.$ In fact, for $\lambda<d,$
there is no constant $C$ such that $\Vert T_{\phi}\Vert\leq C\sup_{z}%
|\phi(z)|$ for all polynomials $\phi.$
\end{proposition}

As we remarked in the introduction, when $\lambda=d,$ the space $H(\mathbb{B}%
^{d},\lambda)$ may be identified with the Hardy space. Thus Properties 1 and 2
in the proposition still hold when $\lambda=d,$ if, say, $\phi$ is continuous
up to the boundary of $\mathbb{B}^{d}$ (or otherwise has a reasonable
extension to the closure of $\mathbb{B}^{d}$).

\begin{proof}
When $\lambda>d,$ the projection operator $P_{\lambda}$ has norm 1 and the
multiplication operator $M_{\phi}$ has norm equal to $\sup_{z}|\phi(z)|$ as an
operator on $L^{2}(\mathbb{B}^{d},\mu_{\lambda})$. Thus, the restriction to
$\mathcal{H}L^{2}(\mathbb{B}^{d},\mu_{\lambda})$ of $P_{\lambda}M_{\phi}$ has
norm at most $\sup_{z}|\phi(z)|.$ Meanwhile, if $\phi$ is non-negative, then
from (\ref{matrix.entry}) we see that $\left\langle f,T_{\phi}f\right\rangle
\geq0$ for all $f\in\mathcal{H}L^{2}(\mathbb{B}^{d},\mu_{\lambda}).$

Let us now assume that $0<\lambda<d.$ Computing on the orthogonal basis in
Theorem \ref{bergmansobolev.thm}, it is a simple exercise to show that%
\begin{equation}
T_{\bar{z}_{j}z_{j}}(z^{m})=\frac{\Gamma(\lambda+|m|)}{m!}\frac{(m+e_{j}%
)!}{\Gamma(\lambda+|m|+1)}z^{m}=\frac{1+m_{j}}{\lambda+|m|}z^{m}. \label{tzm}%
\end{equation}
If we take $\phi(z)=|z|^{2},$ then summing (\ref{tzm}) on $j$ gives%
\[
T_{\phi}z^{m}=\frac{d+|m|}{\lambda+|m|}z^{m}.
\]
Since $\lambda<d,$ this shows that $\left\Vert T_{\phi}\right\Vert >1,$ even
though $\left\vert \phi(z)\right\vert \,<1$ for all $z\in\mathbb{B}^{d}.$
Thus, Property 1 fails for $\lambda<d.$ (From this calculation it easily
follows that if $\phi(z)=(1-\left\vert z\right\vert ^{2})/(\lambda-d),$ then
$T_{\phi}$ is the bounded operator $(\lambda I+N)^{-1},$ for all $\lambda\neq
d.$)

For the second property, we let $\psi(z)=1-\phi(z)=1-|z|^{2}$ which is
positive. From the above calculation we obtain
\[
\langle T_{\psi}z^{m},z^{m}\rangle_{H_{\lambda}}=\Vert z^{m}\Vert_{H_{\lambda
}}^{2}-\left(  \frac{d+|m|}{\lambda+|m|}\right)  \Vert z^{m}\Vert
_{H(\mathbb{B}^{d},\lambda)}^{2},
\]
which is negative if $0<\lambda<d$.

We now show that there is no constant $C$ such that $\Vert T_{\phi}\Vert\leq
C\sup_{z}|\phi(z)|$. Consider
\begin{align*}
\phi_{k}(z)  &  :=(|z|^{2})^{k}=\left(  \sum_{i=1}^{d}|z_{i}|^{2}\right)
^{k}\\
&  =\sum_{|i|=k}\frac{k!}{i!}(|z_{1}|^{2})^{i_{1}}(|z_{2}|^{2})^{i_{2}}%
\cdots(|z_{d}|^{2})^{i_{d}}=\sum_{|i|=k}\frac{k!}{i!}\overline{z}^{i}z^{i}.
\end{align*}
Computing on the orthogonal basis in Theorem \ref{bergmansobolev.thm} we
obtain
\[
T_{\phi_{k}}\mathbf{1}=\sum_{|i|=k}\frac{k!}{i!}(T_{\overline{z}^{i}z^{i}%
}\mathbf{1})=\sum_{|i|=k}\frac{k!}{i!}\frac{i!\Gamma(\lambda)}{\Gamma
(\lambda+k)}\mathbf{1}=\mathcal{I}\frac{k!\Gamma(\lambda)}{\Gamma(\lambda
+k)}\mathbf{1,}%
\]
where $\mathbf{1}$ is the constant function. Here, $\mathcal{I}$ is the number
of multi-indices $i$ of length $d$ such that $|i|=k,$ which is equal to
${\binom{k+d-1}{d-1}}$. Thus
\[
T_{\phi_{k}}\mathbf{1}=\frac{(k+d-1)!}{(d-1)!}\frac{\Gamma(\lambda)}%
{\Gamma(\lambda+k)}\mathbf{1}=\frac{(d+k-1)\cdots(d)}{(\lambda+k-1)\cdots
(\lambda)}\mathbf{1}=\prod_{j=0}^{k-1}\frac{d+j}{\lambda+j}\mathbf{1}.
\]

Consider $\prod_{j=0}^{k-1}\frac{d+j}{\lambda+j}=\prod_{j=0}^{k-1}\left(
1+\frac{d-\lambda}{\lambda+j}\right)  $. Since $d>\lambda,$ the terms
$\frac{d-\lambda}{\lambda+j}$ are positive and $\sum_{j=0}^{\infty}%
\frac{d-\lambda}{\lambda+j}$ diverges. This implies $\prod_{j=0}^{\infty}%
\frac{d+j}{\lambda+j}=\infty$. Since $\sup_{z}|\phi_{k}(z)|=1$ for all $k$,
there is no a constant $C$ such that $\Vert T_{\phi}\Vert\leq C\sup_{z}%
|\phi(z)|$.
\end{proof}

\begin{remark}
For $\lambda<d,$ there does not exist any positive measure $\nu$ on
$\mathbb{B}^{d}$ such that $\left\Vert f\right\Vert _{\lambda}=\left\Vert
f\right\Vert _{L^{2}(\mathbb{B}^{d},\nu)}$ for all $f$ in $H(\mathbb{B}%
^{d},\lambda).$ If such a $\nu$ did exist, then the argument in the first part
of the proof of Proposition \ref{basic.prop} would show that Properties 1 and
2 in the proposition hold.
\end{remark}

\section{Bounded Toeplitz operators\label{bddop}}

In this section, we will consider a class of symbols $\phi$ for which we will
be able to define a Toeplitz operator $T_{\phi}$ as a bounded operator on
$H(\mathbb{B}^{d},\lambda)$ for all $\lambda>0.$ Our definition of $T_{\phi}$
will agree (for the relevant class of symbols) with the usual
\textquotedblleft multiply and project\textquotedblright\ definition for
$\lambda>d.$ In light of the examples in the previous section, we cannot
expect boundedness of $\phi$ to be sufficient to define $T_{\phi}$ as a
bounded operator. Instead, we will consider functions $\phi$ for which $\phi$
and a certain number of derivatives of $\phi$ are bounded.

Our strategy is to use integration by parts to give an alternative expression
for the matrix entries of a Toeplitz operator with sufficiently regular
symbol, in the case $\lambda>d.$ We then take this expression as our
definition of Toeplitz operator in the case $0<\lambda\leq d.$

\begin{theorem}
\label{sobolevtoep.thm}Assume $\lambda>d$ and fix a positive integer $n.$ Let
$\phi$ be a function that is $2n$ times continuously differentiable and for
which $\bar{N}^{k}N^{l}\phi$ is bounded for all $0\leq k,l\leq n.$ Then%
\[
\left\langle f,T_{\phi}g\right\rangle _{\mathcal{H}L^{2}(\mathbb{B}^{d}%
,\mu_{\lambda})}=c_{\lambda+2n}\int_{\mathbb{B}^{d}}C\left[  \left(
\overline{f(z)}\phi(z)g(z)\right)  \right]  \left(  1-|z|^{2}\right)
^{\lambda+2n}d\tau(z)
\]
for all $f,g\in\mathcal{H}L^{2}(\mathbb{B}^{d},\mu_{\lambda}),$ where $C$ is
the operator given by%
\begin{equation}
C=\left(  I+\frac{\bar{N}}{\lambda+2n-1}\right)  \cdots\left(  I+\frac{\bar
{N}}{\lambda+n}\right)  \left(  I+\frac{N}{\lambda+n-1}\right)  \cdots\left(
I+\frac{N}{\lambda}\right)  . \label{cop}%
\end{equation}
Thus, there exist constants $A_{jklm}$ (depending on $n$ and $\lambda$) such
that%
\begin{equation}
\left\langle f,T_{\phi}g\right\rangle _{\mathcal{H}L^{2}(\mathbb{B}^{d}%
,\mu_{\lambda})}=\sum_{j,k,l,m=1}^{n}A_{jklm}\left\langle N^{j}f,\left(
\bar{N}^{k}N^{l}\phi\right)  N^{m}g\right\rangle _{L^{2}(\mathbb{B}^{d}%
,\mu_{\lambda+2n})}. \label{cexpanded}%
\end{equation}

\end{theorem}

\begin{proof}
Assume at first that $f$ and $g$ are polynomials, so that $f$ and $g$ and all
of their derivatives are bounded. We use (\ref{matrix.entry}) and apply the
first equality in Lemma \ref{parts.lem} with $\psi=\bar{f}\phi g.$ We then
apply the first equality in the lemma again with $\psi=(I+N/\lambda)[\bar
{f}\phi g].$ We continue on in this fashion until we have applied the first
equality in Lemma \ref{parts.lem} $n$ times and the second equality $n$ times.
This establishes the desired equality in the case that $f$ and $g$ are
polynomials. For general $f$ and $g$ in $\mathcal{H}L^{2}(\mathbb{B}^{d}%
,\mu_{\lambda}),$ we approximate by sequences $f_{a}$ and $g_{a}$ of
polynomials. From Theorem \ref{shift2n.thm} we can see that convergence of
$f_{a}$ and $g_{a}$ in $\mathcal{H}L^{2}(\mathbb{B}^{d},\mu_{\lambda})$
implies convergence of $N^{j}f_{a}$ and $N^{k}g_{a}$ to $N^{j}f$ and $N^{k}g,$
so that applying (\ref{cexpanded}) to $f_{a}$ and $g_{a}$ and taking a limit
establishes the desired result for $f$ and $g.$
\end{proof}

\begin{definition}
\label{sobolevtoep.def}Assume $0<\lambda\leq d$ and fix a positive integer $n$
such that $\lambda+2n>d.$ Let $\phi$ be a function that is $2n$ times
continuously differentiable and for which $\bar{N}^{k}N^{l}\phi$ is bounded
for all $0\leq k,l\leq n.$ Then we define the Toeplitz operator $T_{\phi}$ to
be the unique bounded operator on $H(\mathbb{B}^{d},\lambda)$ whose matrix
entries are given by%
\begin{equation}
\left\langle f,T_{\phi}g\right\rangle _{H(\mathbb{B}^{d},\lambda)}%
=c_{\lambda+2n}\int_{\mathbb{B}^{d}}C\left[  \left(  \overline{f(z)}%
\phi(z)g(z)\right)  \right]  \left(  1-|z|^{2}\right)  ^{\lambda+2n}dz,
\label{sobolevtoep.form}%
\end{equation}
where $C$ is given by (\ref{cop}).
\end{definition}

Note that from Theorem \ref{bergmansobolev.thm}, $N^{j}f$ and $N^{m}g$ belong
to $L^{2}(\mathbb{B}^{d},\mu_{\lambda+2n})$ for all $0\leq j,m\leq n,$ for all
$f$ and $g$ in $\mathcal{H}L^{2}(\mathbb{B}^{d},\mu_{\lambda}).$ Furthermore,
$\left\Vert N^{j}f\right\Vert _{L^{2}(\mathbb{B}^{d},\mu_{\lambda+2n})}$ and
$\left\Vert N^{m}g\right\Vert _{L^{2}(\mathbb{B}^{d},\mu_{\lambda+2n})}$ are
bounded by constants times $\left\Vert f\right\Vert _{\lambda}$ and
$\left\Vert g\right\Vert _{\lambda}$, respectively. Thus, the right-hand side
of (\ref{sobolevtoep.form}) is a continuous sesquilinear form on
$H(\mathbb{B}^{d},\lambda),$ which means that there is a unique bounded
operator $T_{\phi}$ whose matrix entries are given by (\ref{sobolevtoep.form}).

If $\lambda=d,$ then (as discussed in the introduction) the Hilbert space
$H(\mathbb{B}^{d},\lambda)$ is the Hardy space of holomorphic functions that
are square-integrable over the boundary. In that case, the Toeplitz operator
$T_{\phi}$ will be the zero operator whenever $\phi$ is identically zero on
the boundary of $\mathbb{B}^{d}.$ If $\lambda=d-1,$ $d-2,$ $\ldots,$ then the
inner product on $H(\mathbb{B}^{d},\lambda)$ can be related to the inner
product on the Hardy space. It is not hard to see that in these cases,
$T_{\phi}$ will be the zero operator if $\phi$ and enough of its derivatives
vanish on the boundary of $\mathbb{B}^{d}.$

Let us consider the case in which $\phi(z)=\overline{\psi_{1}(z)}\psi_{2}(z),$
where $\psi_{1}$ and $\psi_{2}$ are holomorphic functions such that the
function and the first $n$ derivatives are bounded. Then when applying $C$ to
$\overline{f(z)}\phi(z)g(z),$ all the $N$-factors go onto the expression $\psi_{2}(z)g(z)$
and all the $\bar{N}$-factors go onto $\overline{f(z)}\overline{\psi_{1}(z)}.$
Recalling from Theorem \ref{bergmansobolev.thm} the formula for the inner
product on $H(\mathbb{B}^{d},\lambda)$, we see that%
\[
\left\langle f,T_{\phi}g\right\rangle _{\mathcal{H}L^{2}(\mathbb{B}^{d}%
,\mu_{\lambda})}=\left\langle \psi_{1}f,\psi_{2}g\right\rangle _{H(\mathbb{B}%
^{d},\lambda)},
\]
as expected. This means that in this case, $T_{\bar{\psi}_{1}\psi_{2}%
}=(M_{\psi_{1}})^{\ast}(M_{\psi_{2}}),$ as in the case $\lambda>d.$ In
particular, Definition \ref{sobolevtoep.def} agrees with the definition we
used in Section \ref{poly.sec} in the case that $\phi$ is a polynomial in $z$
and $\bar{z}.$

\section{Hilbert--Schmidt Toeplitz operators\label{hsop}}

\subsection{Statement of results}

In this section, we will give sufficient conditions under which a Toeplitz
operator $T_{\phi}$ can be defined as a Hilbert--Schmidt operator on
$H(\mathbb{B}^{d},\lambda).$ Specifically, if $\phi$ belongs to $L^{2}%
(\mathbb{B}^{d},\tau)$ then $T_{\phi}$ can be defined as a Hilbert--Schmidt
operator, \textit{provided} that $\lambda>d/2.$ Meanwhile, if $\phi$ belongs
to $L^{1}(\mathbb{B}^{d},\tau),$ then $T_{\phi}$ can be defined as a
Hilbert--Schmidt operator for all $\lambda>0.$ In both cases, we define
$T_{\phi}$ in such a way that for all bounded functions $f$ and $g$ in
$H(\mathbb{B}^{d},\lambda),$ we have%
\begin{equation}
\left\langle f,T_{\phi}g\right\rangle _{\lambda}=c_{\lambda}\int
_{\mathbb{B}^{d}}\overline{f(z)}\phi(z)g(z)(1-|z|^{2})^{\lambda}~d\tau(z),
\label{hsmatrix}%
\end{equation}
where $c_{\lambda}$ is defined by $c_{\lambda}=\Gamma(\lambda)/(\pi^{d}%
\Gamma(\lambda-d)).$ This expression is identical to (\ref{matrix.entry}) in
the case $\lambda>d.$ The value of $c_{\lambda}$ should be interpreted as 0
when $\lambda-d=0,-1,-2,\ldots$. This means that for $\phi$ in $L^{2}%
(\mathbb{B}^{d},\tau)$ or $L^{1}(\mathbb{B}^{d},\tau)$ (but not for other
classes of symbols!), $T_{\phi}$ is the zero operator when $\lambda
=d,d-1,\ldots.$ This strange phenomenon is discussed in the next subsection.
Note that we are \textit{not} claiming $T_{\phi}=0$ for arbitrary symbols when
$\lambda=d,d-1,\ldots,$ but only for symbols that are integrable or
square-integrable with respect to the hyperbolic volume measure $\tau.$ Such
functions must have reasonable rapid decay (in an average sense) near the
boundary of $\mathbb{B}^{d}.$

In the case $\phi\in L^{2}(\mathbb{B}^{d},\tau),$ the restriction
$\lambda>d/2$ is easy to explain: the function $(1-|z|^{2})^{\lambda}$ belongs
to $L^{2}(\mathbb{B}^{d},\tau)$ if and only if $\lambda>d/2.$ Thus, if $f$ and
$g$ are bounded and $\phi$ is in $L^{2}(\mathbb{B}^{d},\tau),$ then
(\ref{hsmatrix}) is absolutely convergent for $\lambda>d/2.$

In this subsection, we state our results; in the next subsection, we discuss
some unusual properties of $T_{\phi}$ for $\lambda<d$; and in the last
subsection of this section we give the proofs.

We begin by considering symbols $\phi$ in $L^{2}(\mathbb{B}^{d},\tau).$

\begin{theorem}
\label{alambda.thm}Fix $\lambda>d/2$ and let $c_{\lambda}=\Gamma(\lambda
)/(\pi^{d}\Gamma(\lambda-d)).$ (We interpret $c_{\lambda}$ to be zero if
$\lambda$ is an integer and $\lambda\leq d$.) Then the operator $A_{\lambda}$
given by%
\[
A_{\lambda}\phi(z)=c_{\lambda}^{2}\int_{\mathbb{B}^{d}}\left[  \frac
{(1-|z|^{2})(1-|w|^{2})}{(1-w\cdot\bar{z})(1-\bar{w}\cdot z)}\right]
^{\lambda}\phi(w)\,d\tau(w)
\]
is a bounded operator from $L^{2}(\mathbb{B}^{d},\tau)$ to itself.
\end{theorem}

\begin{theorem}
\label{hsl2.thm}Fix $\lambda>d/2.$ Then for each $\phi\in L^{2}(\mathbb{B}%
^{d},\tau)$, there is a unique Hilbert--Schmidt operator on $H(\mathbb{B}%
^{d},\lambda),$ denoted $T_{\phi},$ with the property that%
\begin{equation}
\left\langle f,T_{\phi}g\right\rangle _{\lambda}=c_{\lambda}\int
_{\mathbb{B}^{d}}\overline{f(z)}\phi(z)g(z)(1-|z|^{2})^{\lambda}~d\tau(z)
\label{hsmatrix2}%
\end{equation}
for all bounded holomorphic functions $f$ and $g$ in $H(\mathbb{B}^{d}%
,\lambda).$ The Hilbert--Schmidt norm of $T_{\phi}$ is given by%
\[
\left\Vert T_{\phi}\right\Vert _{HS}^{2}=\left\langle \phi,A_{\lambda}%
\phi\right\rangle _{L^{2}(\mathbb{B}^{d},\tau)}.
\]

\end{theorem}

If $\lambda>d$ and $\phi\in L^{2}(\mathbb{B}^{d},\tau)\cap L^{\infty
}(\mathbb{B}^{d},\tau)$, then the definition of $T_{\phi}$ in Theorem
\ref{hsl2.thm} agrees with the \textquotedblleft multiply and
project\textquotedblright\ definition; compare (\ref{matrix.entry}).

Applying Lemma \ref{2lambda.lem} with $\lambda_{1}=\lambda_{2}=\lambda$ and
$\lambda>d/2,$ we see that for all $f$ and $g$ in $H(\mathbb{B}^{d},\lambda),$
the function $z\rightarrow\overline{f(z)}g(z)(1-\left\vert z\right\vert
^{2})^{\lambda}$ is in $L^{2}(\mathbb{B}^{d},\tau).$ This means that the
integral on the right-hand side of (\ref{hsmatrix2}) is absolutely convergent
for all $f,g\in H(\mathbb{B}^{d},\lambda).$ It is then not hard to show that
(\ref{hsmatrix2}) holds for all $f,g\in H(\mathbb{B}^{d},\lambda).$

The operator $A_{\lambda}$ coincides, up to a constant, with the Berezin
transform. Let $\chi_{z}^{\lambda}(w):=K_{\lambda}(z,w)$ be the coherent state
at the point $z,$ which satisfies $f(z)=\left\langle \chi_{z}^{\lambda
},f\right\rangle _{\lambda}$ for all $f\in H(\mathbb{B}^{d},\lambda).$ Then
one standard definition of the Berezin transform $B_{\lambda}$ is%
\[
B_{\lambda}\phi=\frac{\left\langle \chi_{z}^{\lambda},T_{\phi}\chi
_{z}^{\lambda}\right\rangle _{\lambda}}{\left\langle \chi_{z}^{\lambda}%
,\chi_{z}^{\lambda}\right\rangle _{\lambda}}.
\]
The function $B_{\lambda}\phi$ may be thought of as the Wick-ordered symbol of
$T_{\phi},$ where $T_{\phi}$ is thought of as the \textit{anti}-Wick-ordered
quantization of $\phi.$ Using the formula (Theorem \ref{bergmansobolev.thm})
for the reproducing kernel along with (\ref{hsmatrix2}), we see that
$A_{\lambda}=c_{\lambda}B_{\lambda}.$ (Note that $\chi_{z}^{\lambda}(w)$ is a
bounded function of $w$ for each fixed $z\in\mathbb{B}^{d}$ and that
$\left\langle \chi_{z}^{\lambda},\chi_{z}^{\lambda}\right\rangle _{\lambda
}=K_{\lambda}(z,z).$)

Note that $\tau$ is an infinite measure, which means that if $\phi$ is in
$L^{2}(\mathbb{B}^{d},\tau)$ or $L^{1}(\mathbb{B}^{d},\tau),$ then $\phi$ must
tend to zero at the boundary of $\mathbb{B}^{d},$ at least in an average
sense. This decay of $\phi$ is what allows (\ref{hsmatrix2}) to be a
convergent integral. If, for example, we want to take $\phi(z)\equiv1,$ then
we cannot use (\ref{hsmatrix2}) to define $T_{\phi},$ but must instead use the
definition from Section \ref{poly.sec} or Section \ref{bddop}.

Note also that the space of Hilbert--Schmidt operators on $H(\mathbb{B}%
^{d},\lambda)$ may be viewed as the quantum counterpart of $L^{2}%
(\mathbb{B}^{d},\tau).$ It is thus natural to investigate the question of when
the Berezin--Toeplitz quantization maps $L^{2}(\mathbb{B}^{d},\tau)$ into the
Hilbert--Schmidt operators.

We now show that if one considers a symbol $\phi$ in $L^{1}(\mathbb{B}%
^{d},\tau),$ then one obtains a Hilbert--Schmidt Toeplitz operator $T_{\phi}$
for all $\lambda>0.$

\begin{theorem}
\label{hsl1.thm}Fix $\lambda>0$ and let $c_{\lambda}$ be as in Theorem
\ref{hsl2.thm}. Then for each $\phi\in L^{1}(\mathbb{B}^{d},\tau),$ there
exists a unique Hilbert--Schmidt operator on $H(\mathbb{B}^{d},\lambda),$
denoted $T_{\phi},$ with the property that%
\begin{equation}
\left\langle f,T_{\phi}g\right\rangle _{\lambda}=c_{\lambda}\int
_{\mathbb{B}^{d}}\overline{f(z)}\phi(z)g(z)(1-|z|^{2})^{\lambda}~d\tau(z)
\label{hsmatrix3}%
\end{equation}
for all bounded holomorphic functions $f$ and $g$ in $H(\mathbb{B}^{d}%
,\lambda).$ The Hilbert--Schmidt norm of $T_{\phi}$ satisfies%
\[
\left\Vert T_{\phi}\right\Vert _{HS}\leq c_{\lambda}\left\Vert \phi\right\Vert
_{L^{1}(\mathbb{B}^{d},\tau)}.
\]

\end{theorem}

Using the pointwise bounds on elements of $H(\mathbb{B}^{d},\lambda)$ coming
from the reproducing kernel, we see immediately that for all $f,g\in
H(\mathbb{B}^{d},\lambda),$ the function $z\rightarrow\overline{f(z)}%
g(z)(1-|z|^{2})^{\lambda}$ is bounded. It is then not hard to show that
(\ref{hsmatrix3}) holds for all $f,g\in H(\mathbb{B}^{d},\lambda).$

We have already remarked that the definition of $T_{\phi}$ given in this
section agrees with the \textquotedblleft multiply and
project\textquotedblright\ definition when $\lambda>d$ (and $\phi$ is
bounded). It is also easy to see that the definition of $T_{\phi}$ given in
this section agrees with the one in Section \ref{bddop}, when $\phi$ falls
under the hypotheses of both Definition \ref{sobolevtoep.def} and either
Theorem \ref{hsl2.thm} or Theorem \ref{hsl1.thm}. For some positive integer
$n,$ consider the set of $\lambda$'s for which $\lambda+2n>d$ and
$\lambda>d/2,$ i.e., $\lambda>\max(d-2n,d/2).$ Now suppose that $\phi$ belongs
to $L^{2}(\mathbb{B}^{d},\tau)$ and that $N^{k}\bar{N}^{l}\phi$ is bounded for
all $0\leq k,l\leq n.$ It is easy to see that the matrix entries $\left\langle
f,T_{\phi}g\right\rangle _{\lambda}$ depend real-analytically on $\lambda$ for
fixed polynomials $f$ and $g,$ whether $T_{\phi}$ is defined by Definition
\ref{sobolevtoep.def} or by Theorem \ref{hsl2.thm}. For $\lambda>d,$ the two
matrix entries agree because both definitions of $T_{\phi}$ agree with the
\textquotedblleft multiply and project\textquotedblright\ definition. The
matrix entries therefore must agree for all $\lambda>\max(d-2n,d/2).$ Since
polynomials are dense in $H(\mathbb{B}^{d},\lambda)$ and both definitions of
$T_{\phi}$ give bounded operators, the two definitions of $T_{\phi}$ agree.
The same reasoning shows agreement of Definition \ref{sobolevtoep.def} and
Theorem \ref{hsl1.thm}.

\subsection{Discussion}

Before proceeding on with the proof, let us make a few remarks about the way
we are defining Toeplitz operators in this section. For $\lambda>d,$
$c_{\lambda}$ is the normalization constant that makes the measure
$\mu_{\lambda}$ a probability measure, which can be computed to have the value
$\Gamma(\lambda)/(\pi^{d}\Gamma(\lambda-d)).$ For $\lambda\leq d,$ although
the measure $(1-|z|^{2})^{\lambda}~d\tau(z)$ is an infinite measure, we simply
use the same formula for $c_{\lambda}$ in terms of the gamma function. We
understand this to mean that $c_{\lambda}=0$ whenever $\lambda$ is an integer
in the range $(0,d].$ It also means that $c_{\lambda}$ is negative when
$d-1<\lambda<d$ and when $d-3<\lambda<d-2,$ etc.

In the cases where $c_{\lambda}=0$, we have that $T_{\phi}=0$ for all $\phi$
in $L^{1}(\mathbb{B}^{d},\tau)$ or $L^{2}(\mathbb{B}^{d},\tau).$ This first
occurs when $\lambda=d.$ Recall that for $\lambda=d,$ the space $H(\mathbb{B}%
^{d},\lambda)$ can be identified with the Hardy space of holomorphic functions
square-integrable over the boundary. Meanwhile, having $\phi$ being integrable
or square-integrable with respect to $\tau$ means that $\phi$ tends to zero
(in an average sense) at the boundary, in which case it is reasonable that
$T_{\phi}$ should be zero as an operator on the Hardy space. For other integer
values of $\lambda\leq d,$ the inner product on $H(\mathbb{B}^{d},\lambda)$
can be expressed using the methods of Section \ref{bergmansobolev.sec} in
terms of integration over the boundary, but involving the functions and their
derivatives. In that case, we expect $T_{\phi}$ to be zero if $\phi$ has
sufficiently rapid decay at the boundary, and it is reasonable to think that
having $\phi$ in $L^{1}$ or $L^{2}$ with respect to $\tau$ constitutes
sufficiently rapid decay. Note, however, that the conclusion that $T_{\phi}=0$
when $c_{\lambda}=0$ applies \textit{only} when $\phi$ is in $L^{1}$ or
$L^{2}$; for other classes of symbols, such as polynomials, $T_{\phi}$ is not
necessarily zero. For example, $T_{z^{m}}$ is equal to $M_{z^{m}},$ which is
certainly a nonzero operator on $H(\mathbb{B}^{d},\lambda),$ for all
$\lambda>0.$

Meanwhile, if $c_{\lambda}<0,$ then we have the curious situation that if
$\phi$ is positive and in $L^{1}$ or $L^{2}$ with respect to $\tau,$ then the
operator $T_{\phi}$ is actually a \textit{negative} operator. This is merely a
dramatic example of a phenomenon we have already noted: for $\lambda<d,$
non-negative symbols do not necessarily give rise to non-negative Toeplitz
operators. Again, though, the conclusion that $T_{\phi}$ is negative for
$\phi$ positive applies only when $\phi$ belongs to $L^{1}$ or $L^{2}.$ For
example, the constant function $\mathbf{1}$ always maps to the (positive!)
identity operator, regardless of the value of $\lambda.$

\subsection{Proofs}

As motivation, we begin by computing the Hilbert--Schmidt norm of Toeplitz
operators in the case $\lambda>d.$ For any bounded measurable $\phi,$ we
extend the Toeplitz operator $T_{\phi}$ to all of $L^{2}(\mathbb{B}^{d}%
,\mu_{\lambda})$ by making it zero on the orthogonal complement of the
holomorphic subspace. This extension is given by the formula $P_{\lambda
}M_{\phi}P_{\lambda}.$ Then the Hilbert--Schmidt norm of the operator
$T_{\phi}$ on $\mathcal{H}L^{2}(\mathbb{B}^{d},\mu_{\lambda})$ is the same as
the Hilbert--Schmidt norm of the operator $P_{\lambda}M_{\phi}P_{\lambda}$ on
$L^{2}(\mathbb{B}^{d},\mu_{\lambda}).$ Since $P_{\lambda}$ is computed as
integration against the reproducing kernel, we may compute that%
\[
P_{\lambda}M_{\phi}P_{\lambda}f(z)=\int_{\mathbb{B}^{d}}\mathcal{K}_{\phi
}(z,w)f(w)\,d\mu_{\lambda}(w),
\]
where%
\[
\mathcal{K}_{\phi}(z,w)=\int_{\mathbb{B}^{d}}K(z,u)\phi(u)K(u,w)\,d\mu
_{\lambda}(u).
\]

If we can show that $\mathcal{K}_{\phi}$ is in $L^{2}(\mathbb{B}^{d}%
\times\mathbb{B}^{d},\mu_{\lambda}\times\mu_{\lambda})$, then it will follow
by a standard result that $T_{\phi}$ is Hilbert--Schmidt, with
Hilbert--Schmidt norm equal to the $L^{2}$ norm of $\mathcal{K}_{\phi}.$ For
sufficiently nice $\phi,$ we can compute the $L^{2}$ norm of $\mathcal{K}%
_{\phi}$ by rearranging the order of integration and using twice the
reproducing identity $\int K(z,w)K(w,u)~d\mu_{\lambda}(w)=K(z,u).$ (This
identity reflects that $P_{\lambda}^{2}=P_{\lambda}.$) This yields%
\[
\int_{\mathbb{B}^{d}\times\mathbb{B}^{d}}\left\vert \mathcal{K}_{\phi
}(z,w)\right\vert ^{2}~d\mu_{\lambda}(z)~d\mu_{\lambda}=\left\langle
\phi,A\phi\right\rangle _{L^{2}(\mathbb{B}^{d},\tau)},
\]
where $A_{\lambda}$ is the integral operator given by%
\begin{align}
A_{\lambda}\phi(z)  &  =c_{\lambda}^{2}\int_{\mathbb{B}^{d}}\left\vert
K(z,w)\right\vert ^{2}(1-|z|^{2})^{\lambda}(1-|w|^{2})^{\lambda}\phi
(w)~d\tau(w)\nonumber\\
&  =c_{\lambda}^{2}\int_{\mathbb{B}^{d}}\left[  \frac{(1-|z|^{2})(1-|w|^{2}%
)}{(1-\bar{w}\cdot z)(1-\bar{z}\cdot w)}\right]  ^{\lambda}\phi(w)\,d\tau(w).
\label{alambda.form1}%
\end{align}

In the case $d/2<\lambda\leq d,$ it no longer makes sense to express $T_{\phi
}$ as $P_{\lambda}M_{\phi}P_{\lambda}.$ Nevertheless, we can consider an
operator $A_{\lambda}$ defined by (\ref{alambda.form1}). Our goal is to show
that for all $\lambda>d/2,$ (1) $A_{\lambda}$ is a bounded operator on
$L^{2}(\mathbb{B}^{d},\tau)$ and (2) if we define $T_{\phi}$ by
(\ref{hsmatrix}), then the Hilbert--Schmidt norm of $T_{\phi}$ is given by
$\left\langle \phi,A_{\lambda}\phi\right\rangle _{L^{2}(\mathbb{B}^{d},\tau
)}.$ We will obtain similar results for all $\lambda>0$ if $\phi\in
L^{1}(\mathbb{B}^{d},\tau).$

\begin{proof}[Proof of Theorem \ref{alambda.thm}]
We give two proofs of this result; the first
generalizes more easily to other bounded symmetric domains, whereas the second
relates $A_{\lambda}$ to the Laplacian for $\mathbb{B}^{d}$ (compare \cite{E}).

\textit{First Proof.} We let
\[
F_{\lambda}(z,w)=c_{\lambda}^{2}\left[  \frac{(1-|z|^{2})(1-|w|^{2})}%
{(1-\bar{w}\cdot z)(1-\bar{z}\cdot w)}\right]  ^{\lambda}\text{;}%
\]
i.e., $F_{\lambda}$ is the integral kernel of the operator $A_{\lambda}.$ A
key property of $F_{\lambda}$ is its invariance under automorphisms:
$F_{\lambda}(\psi(z),\psi(w))=F_{\lambda}(z,w)$ for each automorphism
(biholomorphism) $\psi$ of $\mathbb{B}^{d}$ and all $z,w\in\mathbb{B}^{d}.$ To
establish the invariance of $F_{\lambda},$ let
\begin{equation}
f_{\lambda}(z)=c_{\lambda}^{2}(1-|z|^{2})^{\lambda}. \label{flambda.def}%
\end{equation}
According to Lemma 1.2 of \cite{Z}, $F_{\lambda}(z,w)=f_{\lambda}(\phi
_{w}(z)),$ where $\phi_{w}$ is an automorphism of $\mathbb{B}^{d}$ taking $0$
to $w$ and satisfying $\phi_{w}^{2}=I.$ Now, if $\psi$ is any automorphism,
the classification of automorphisms (Theorem 1.4 of \cite{Z}) implies that
$\psi\circ\phi_{w}=\phi_{\psi(w)}\circ U$ for some unitary matrix $U.$ From
this we can obtain $\phi_{\psi(w)}=U\circ\phi_{w}\circ\psi^{-1},$ and so%
\[
f_{\lambda}(\phi_{\psi(w)}(\psi(z)))=f_{\lambda}(U(\phi_{w}(\psi^{-1}%
(\psi(z))))=f_{\lambda}(\phi_{w}(z)),
\]
i.e., $F_{\lambda}(\psi(z),\psi(w))=F_{\lambda}(z,w).$

The invariance of $F_{\lambda}$ under automorphisms means that $A_{\lambda
}\phi$ can be thought of as a convolution (over the automorphism group
$PSU(d,1)$) of $\phi$ with the function $f_{\lambda}.$ What this means is that%
\[
A_{\lambda}\phi(z)=\int_{G}f_{\lambda}(gh^{-1}\cdot0)\phi(h\cdot0)~dh,
\]
where $g\in G$ is chosen so that $g\cdot0=z.$ Here $G=PSU(d,1)$ is the group
of automorphisms of $\mathbb{B}^{d}$ (given by fractional linear
transformations) and $dh$ is an appropriately normalized Haar measure on $G.$
Furthermore, $L^{2}(\mathbb{B}^{d},\tau)$ can be identified with the
right-$K$-invariant subspace of $L^{2}(G,dg)$, where $K:=U(d)$ is the
stabilizer of $0.$

If $\lambda>d,$ then $f_{\lambda}$ is in $L^{1}(\mathbb{B}^{d},\tau),$ in
which case it is easy to prove that $A_{\lambda}$ is bounded; see, for
example, Theorem 2.4 in \cite{B1}. This argument does not work if $\lambda\leq
d.$ Nevertheless, if $\lambda>d/2,$ an easy computation shows that
$f_{\lambda}$ belongs to $L^{2}(\mathbb{B}^{d},\tau)$ and also to
$L^{p}(\mathbb{B}^{d},\tau)$ for some $p<2.$ We could at this point appeal to
a general result known as the Kunze--Stein phenomenon \cite{KS}. The result
states that on connected semisimple Lie groups $G$ with finite center
(including $PSU(d,1)$), convolution with a function in $L^{p}(G,dg),$ $p<2,$
is a bounded operator from $L^{2}(G,dg)$ to itself. (See \cite{Cow} for a
proof in this generality.) However, the proof of this result is simpler in the
case we are considering, where the function in $L^{p}(G,dg)$ is bi-$K$%
-invariant and the other function is right-$K$-invariant. (In our case, the
function in $L^{p}(G,dg)$ is the function $g\rightarrow f_{\lambda}(g\cdot0)$
and the function in $L^{2}(G,dg)$ is $g\rightarrow\phi(g\cdot0).$) Using the
Helgason Fourier transform along with its behavior under convolution with a
bi-$K$-invariant function (\cite[Lemma III.1.4]{He3}), we need only show that
the spherical Fourier transform of $f_{\lambda}$ is bounded. (Helgason proves
Lemma III.1.4 under the assumption that the functions are continuous and of
compact support, but the proof also applies more generally.) Meanwhile,
standard estimates show that for every $\varepsilon>0,$ the spherical
functions are in $L^{2+\varepsilon}(G/K),$ with $L^{2+\varepsilon}(G/K)$ norm
bounded independent of the spherical function. (Specifically, in the notation
of \cite[Sect. IV.4]{He2}, for all $\lambda\in\mathfrak{a}^{\ast},$ we have
$\left\vert \phi_{\lambda}(g)\right\vert \leq\phi_{0}(g),$ and estimates on
$\phi_{0}$ (e.g., \cite[Prop. 2.2.12]{AJ}) show that $\phi_{0}$ is in
$L^{2+\varepsilon}$ for all $\varepsilon>0.$)

Choosing $\varepsilon$ so that $1/p+1/(2+\varepsilon)=1$ establishes the
desired boundedness.

\textit{Second proof. }If $c_{\lambda}=0$ (i.e., if $\lambda\in\mathbb{Z}$ and
$\lambda\leq d$), then there is nothing to prove. Thus we assume $c_{\lambda}$
is nonzero, in which case $c_{\lambda+1}$ is also nonzero. The invariance of
$F_{\lambda}$ under automorphisms together with the square-integrability of
the function $(1-|z|^{2})^{\lambda}$ for $\lambda>d/2$ show that the integral
defining $A_{\lambda}f(z)$ is absolutely convergent for all $z.$

We introduce the (hyperbolic) Laplacian $\Delta$ for $\mathbb{B}^{d},$ given
by%
\begin{equation}
\Delta=(1-|z|^{2})\sum_{j,k=1}^{d}(\delta_{jk}-\bar{z}_{j}z_{k})\frac
{\partial^{2}}{\partial\bar{z}_{j}\partial z_{k}}. \label{laplace}%
\end{equation}
(This is a \textit{negative} operator.) This operator commutes with the
automorphisms of $\mathbb{B}^{d}.$ It is known (e.g., \cite{Str}) that
$\Delta$ is an unbounded self-adjoint operator on $L^{2}(\mathbb{B}^{d}%
,\tau),$ on the domain consisting of those $f$'s in $L^{2}(\mathbb{B}^{d}%
,\tau)$ for which $\Delta f$ in the distribution sense belongs to
$L^{2}(\mathbb{B}^{d},\tau).$ In particular, if $f\in L^{2}(\mathbb{B}%
^{d},\tau)$ is $C^{2}$ and $\Delta f$ in the ordinary sense belongs to
$L^{2}(\mathbb{B}^{d},\tau),$ then $f\in Dom(\Delta).$

We now claim that%
\begin{equation}
\Delta_{z}F_{\lambda}(z,w)=\lambda(\lambda-d)(F_{\lambda}(z,w)-F_{\lambda
+1}(z,w)), \label{laplacef}%
\end{equation}
where $\Delta_{z}$ indicates that $\Delta$ is acting on the variable $z$ with
$w$ fixed. Since $\Delta$ commutes with automorphisms, it again suffices to
check this when $w=0,$ in which case it is a straightforward algebraic
calculation. Suppose, then, that $\phi$ is a $C^{\infty}$ function of compact
support. In that case, we are free to differentiate under the integral to
obtain%
\begin{equation}
\Delta A_{\lambda}\phi=\lambda(\lambda-d)A_{\lambda}\phi-\lambda
(\lambda-d)A_{\lambda+1}\phi. \label{laplace.ident1}%
\end{equation}

Now, the invariance of $F_{\lambda}$ tells us that $L^{2}(\mathbb{B}^{d}%
,\tau)$ norm of $F_{\lambda}(z,w)$ as a function of $z$ is finite for all $w$
and independent of $w.$ Putting the $L^{2}$ norm inside the integral then
shows that $A_{\lambda}\phi$ and $A_{\lambda+1}\phi$ are in $L^{2}%
(\mathbb{B}^{d},\tau).$ This shows that $A_{\lambda}\phi$ is in $Dom(\Delta).$
Furthermore, the condition $\lambda>d/2$ implies that $\lambda(\lambda
-d/2)>-d^{2}/4.$ It is known that the $L^{2}$ spectrum of $\Delta$ is
$(-\infty,-d^{2}/4].$ For general symmetric space of the noncompact type, the
$L^{2}$ spectrum of the Laplacian is $(-\infty,-\left\Vert \rho\right\Vert
^{2}],$ where $\rho$ is half the sum of the positive (restricted) roots for
$G/K,$ counted with their multiplicity. In our case, there is one positive
root $\alpha$ with multiplicity $(2d-2)$ and another positive root $2\alpha$
with multiplicity 1. (See the entry for \textquotedblleft A
IV\textquotedblright\ in Table VI of Chapter X of \cite{He1}.) Thus,
$\rho=d\alpha.$ It remains only to check that if the metric is normalized so
that the Laplacian comes out as in (\ref{laplace}), then $\left\Vert
\alpha\right\Vert ^{2}=1/4.$ This is a straightforward but unilluminating
computation, which we omit.

Since $\lambda(\lambda-d)$ is in the resolvent set of $\Delta,$ we may rewrite
(\ref{laplace.ident1}) as%
\[
A_{\lambda}\phi=-\lambda(\lambda-d)[\Delta-\lambda(\lambda-d)I]^{-1}%
A_{\lambda+1}\phi.
\]
Suppose now that $\lambda+1>d,$ so that (as remarked above) $A_{\lambda+1}$ is
bounded. Since $[\Delta-cI]^{-1}$ is a bounded operator for all $c$ in the
resolvent of $\Delta,$ we see that $A_{\lambda}$ has a bounded extension from
$C_{c}^{\infty}(\mathbb{B}^{d})$ to $L^{2}(\mathbb{B}^{d},\tau).$ Since the
integral computing $A_{\lambda}\phi(z)$ is a continuous linear functional on
$L^{2}(\mathbb{B}^{d},\tau)$ (integration against an element of $L^{2}%
(\mathbb{B}^{d},\tau)$), it is easily seen that this bounded extension
coincides with the original definition of $A_{\lambda}.$

The above argument shows that $A_{\lambda}$ is bounded if $\lambda>d/2$ and
$\lambda+1>d.$ Iteration of the argument then shows boundedness for all
$\lambda>d/2.$
\end{proof}

\begin{proof}[Proof of Theorem \ref{hsl2.thm}]
We wish to show that for all $\lambda>d/2,$ if
$\phi$ is in $L^{2}(\mathbb{B}^{d},\tau),$ then there is a unique
Hilbert--Schmidt operator $T_{\phi}$ with matrix entries given in
(\ref{hsmatrix}) for all polynomials, and furthermore, $\left\Vert T_{\phi
}\right\Vert _{HS}^{2}=\left\langle \phi,A_{\lambda}\phi\right\rangle
_{\lambda}.$ At the beginning of this section, we had an calculation of
$\left\Vert T_{\phi}\right\Vert $ in terms of $A_{\lambda},$ but this argument
relied on writing $T_{\phi}$ as $P_{\lambda}M_{\phi}P_{\lambda},$ which does
not make sense for $\lambda\leq d.$

We work with an orthonormal basis for $H(\mathbb{B}^{d},\lambda)$ consisting
of normalized monomials, namely,
\[
e_{m}(z)=z^{m}\sqrt{\frac{\Gamma(\lambda+|m|)}{m!\Gamma(\lambda)}},
\]
for each multi-index $m.$ Then we want to establish the existence of a
Hilbert--Schmidt operator whose matrix entries in this basis are given by%
\begin{equation}
a_{lm}:=c_{\lambda}\int_{\mathbb{B}^{d}}\overline{e_{l}(z)}\phi(z)e_{m}%
(z)(1-|z|^{2})^{\lambda}~d\tau(z). \label{alm}%
\end{equation}
There will exist a unique such operator provided that $\sum_{l,m}|a_{lm}%
|^{2}<\infty.$

If we assume, for the moment, that Fubini's Theorem applies, we obtain%
\begin{align}
&  \sum_{l,m}\left\vert a_{lm}\right\vert ^{2}\nonumber\\
&  =c_{\lambda}^{2}\int_{\mathbb{B}^{d}}\int_{\mathbb{B}^{d}}\sum_{l,m}%
\frac{\Gamma(\lambda+|l|)}{l!\Gamma(\lambda)}\frac{\Gamma(\lambda
+|m|)}{m!\Gamma(\lambda)}\bar{z}^{l}w^{l}z^{m}\bar{w}^{m}\nonumber\\
&  \times\phi(z)\overline{\phi(w)}(1-|z|^{2})^{\lambda}(1-|w|^{2})^{\lambda
}~d\tau(z)~d\tau(w), \label{normsum}%
\end{align}
where $l$ and $m$ range over all multi-indices of length $d.$

We now apply the binomial series
\[
\frac{1}{(1-r)^{\lambda}}=\sum_{k=0}^{\infty}\binom{\lambda+k-1}{k}r^{k}%
\]
for $r\in\mathbb{C}$ with $\left\vert r\right\vert <1,$ where%
\[
\binom{\lambda+k-1}{k}=\frac{\Gamma(\lambda+k)}{k!\Gamma(\lambda)}.
\]
(This is the so-called negative binomial series.) We apply this with
$r=\sum_{j}\bar{z}_{j}w_{j},$ and we then apply the (finite) multinomial
series to the computation of $(\bar{z}\cdot w)^{k}.$ The result is that%
\begin{equation}
\sum_{l}\frac{\Gamma(\lambda+|l|)}{l!\Gamma(\lambda)}\bar{z}^{l}w^{l}=\frac
{1}{(1-\bar{z}\cdot w)^{\lambda}}, \label{kernelSeries}%
\end{equation}
where the sum is over all multi-indices $l.$ Applying this result,
(\ref{normsum}) becomes
\begin{equation}
\sum_{l,m}\left\vert a_{lm}\right\vert ^{2}=\left\langle \phi,A_{\lambda}%
\phi\right\rangle _{\lambda}, \label{normsum2}%
\end{equation}
which is what we want to show.

Assume at first that $\phi$ is \textquotedblleft nice,\textquotedblright\ say,
continuous and supported in a ball of radius $r<1.$ This ball has finite
measure and $\phi$ is bounded on it. Thus, if we put absolute values inside
the sum and integral on the right-hand side of (\ref{normsum}), finiteness of
the result follows from the absolute convergence of the series
(\ref{kernelSeries}). Thus, Fubini's Theorem applies in this case.

Now for a general $\phi\in L^{2}(\mathbb{B}^{d},\tau),$ choose $\phi_{j}$
converging to $\phi$ with $\phi_{j}$ \textquotedblleft nice.\textquotedblright%
\ Then (\ref{normsum2}) tells us that $T_{\phi_{j}}$ is a Cauchy sequence in
the space of Hilbert--Schmidt operators, which therefore converges in the
Hilbert--Schmidt norm to some operator $T.$ The matrix entries of $T_{\phi
_{j}}$ in the basis $\{e_{m}\}$ are by construction given by the integral in
(\ref{alm}). The matrix entries of $T$ are the limit of the matrix entries of
$T_{\phi_{j}},$ hence also given by (\ref{alm}), because $e_{l}$ and $e_{m}$
are bounded and $(1-|z|^{2})^{\lambda}$ belongs to $L^{2}(\mathbb{B}^{d}%
,\tau)$ for $\lambda>d/2.$

We can now establish that (\ref{hsmatrix2}) in Theorem \ref{hsl2.thm} holds
for all bounded holomorphic functions $f$ and $g$ in $H(\mathbb{B}^{d}%
,\lambda)$ by approximating these functions by polynomials.
\end{proof}

\begin{proof}[Proof of Theorem \ref{hsl1.thm}]
In the proof of Theorem \ref{hsl2.thm}, we did not
use the assumption $\lambda>d/2$ until the step in which we approximated
arbitrary functions in $L^{2}(\mathbb{B}^{d},\tau)$ by \textquotedblleft
nice\textquotedblright\ functions. In particular, if $\phi$ is nice, then
(\ref{alm}) makes sense for all $\lambda>0,$ and (\ref{normsum2}) still holds.
Now, since $F_{\lambda}(z,w)=f_{\lambda}(\phi_{w}(z)),$ where $f_{\lambda}$ is
given by (\ref{flambda.def}), we see that $\left\vert F_{\lambda
}(z,w)\right\vert \leq c_{\lambda}^{2}$ for all $z,w\in\mathbb{B}^{d}.$ Thus,
\[
\left\langle \phi,A_{\lambda}\phi\right\rangle _{\lambda}\leq c_{\lambda}%
^{2}\left\Vert \phi\right\Vert _{L^{1}(\mathbb{B}^{d},\tau)}^{2}%
\]
for all nice $\phi.$ An easy approximation argument then establishes the
existence of a Hilbert--Schmidt operator with the desired matrix entries for
all $\phi\in L^{1}(\mathbb{B}^{d},\tau),$ with the desired estimate on the
Hilbert--Schmidt norm.
\end{proof}

\end{document}